\documentclass[reqno]{amsart}
\usepackage{amssymb}
\usepackage{amsfonts}
\usepackage{graphicx}
\usepackage{amssymb}
\usepackage{amsmath}
\usepackage{bm}
\usepackage{color}
\usepackage{enumerate}
\usepackage{subfigure}
\usepackage{epstopdf}
\usepackage{setspace}
\usepackage{amssymb}
\usepackage{amsmath}
\usepackage{caption}

\newtheorem{theorem}{Theorem}
\theoremstyle{plain}

\newtheorem{claim}{Claim}

\newtheorem{example}{Example}

\newtheorem{lemma}{Lemma}

\numberwithin{equation}{section}

\begin{document}
\title[Arithmetic representations of real numbers]{Arithmetic
representations of real numbers in terms of self-similar sets}

\begin{abstract}
Suppose $n\geq 2$ and $\mathcal{A}_{i}\subset \{0,1,\cdots ,(n-1)\}$ for $%
i=1,\cdots ,l,$ let $K_{i}=\bigcup\nolimits_{a\in \mathcal{A}%
_{i}}n^{-1}(K_{i}+a)$ be self-similar sets contained in $[0,1].$ Given $%
m_{1},\cdots ,m_{l}\in \mathbb{Z}$ with $\prod\nolimits_{i}m_{i}\neq 0,$ we
let%
\begin{equation*}
S_{x}=\left\{ \mathbf{(}y_{1},\cdots ,y_{l}\mathbf{)}:m_{1}y_{1}+\cdots
+m_{l}y_{l}=x\text{ with }y_{i}\in K_{i}\text{ }\forall i\right\} .
\end{equation*}%
In this paper, we analyze the Hausdorff dimension and Hausdorff measure of
the following set
\begin{equation*}
U_{r}=\{x:\mathbf{Card}(S_{x})=r\},
\end{equation*}%
where $\mathbf{Card}(S_{x})$ denotes the cardinality of $S_{x}$, and $r\in
\mathbb{N}^{+}$. We prove under the so-called covering condition that the
Hausdorff dimension of $U_{1}$ can be calculated in terms of some matrix.
Moreover, if $r\geq 2$, we also give some sufficient conditions such that
the Hausdorff dimension of $U_{r}$ takes only finite values, and these
values can be calculated explicitly. Furthermore, we come up with some
sufficient conditions such that the dimensional Hausdorff measure of $U_{r}$
is infinity. Various examples are provided. Our results can be viewed as the
exceptional results for the classical slicing problem in geometric measure
theory.
\end{abstract}
\author{kan Jiang}
\address{Department of Mathematics, Ningbo University, Ningbo 315211, P. R.
China}
\email{jiangkan@nbu.edu.cn; kanjiangbunnik@yahoo.com}
\author{lifeng Xi}
\address{Department of Mathematics, Ningbo University, Ningbo 315211, P. R.
China}
\email{xilifeng@nbu.edu.cn;  xilifengningbo@yahoo.com}
\thanks{Lifeng Xi is the corresponding author. The work is supported by National Natural Science Foundation of China (Nos.
11831007, 11771226, 11701302, 11371329, 11471124, 11671147). The work is
also supported by K.C. Wong Magna Fund in Ningbo University.}
\subjclass[2000]{Primary 28A80}
\keywords{Fractal, self-similar set, unique representation, section, projection}

\maketitle

\section{Introduction}

Representation of real numbers is a topic of great interest in number
theory. There are many approaches which can represent real numbers, for
instance, the $\beta$-expansions \cite{KO,Akiyama,SN,BakerG,KM,Renyi}, the
continued fractions \cite{KarmaOomen,KCor}, multiplication (division,
quadratic sum) on fractal sets \cite{XiKan1,Tyson}, the {L}\"uroth
expansions \cite{KarmaCor1}, and so forth. These representations are related
to many different mathematical aspects, for instance, the matrix theory,
ergodic theory, fractal geometry, Diophantine approximation, combinatorics,
and so on. Different representations have distinct properties from various
aspects. Investigating the ergodic, fractal and combinatorial properties of
these representations consists of one of the main directions in number
theory.

Expansions in non-integer bases were pioneered by R\'{e}nyi \cite{Renyi} and
Parry \cite{Parry}. Let $1<\beta <2$. Given any $x\in \lbrack 0,(\beta
-1)^{-1}]$, a sequence $(a_{n})\in \{0,1\}^{\mathbb{\infty }}$ is called a $%
\beta $-expansion of $x$ if $x=\sum_{n=1}^{\infty }a_{n}\beta ^{-n}.$
Sidorov \cite{Sidorov} proved that given any $1<\beta <2$, then almost every
point in $[0,(\beta -1)^{-1}]$ has uncountably many expansions. In fact,
Sidorov \cite{Sidorov2003}, Dajani and de Vries \cite{DajaniDeVrie} proved
that given any $1<\beta <2$, then Lebesgue almost every point has
uncountably many universal expansions. Erd\H{o}s et al. \cite{EHJ} proved
that there exist some $1<\beta <2$ and $x\in \lbrack 0,(\beta -1)^{-1}]$
such that $x$ has precisely $k$ different $\beta $-expansions. However, up
to now, there are few results concerning with the set of points with exactly
$k$ different $\beta $-expansions, see \cite{DJKL, KarmaKan2} and references
therein. In particular, if $x$ has a unique expansion (such point is called
the univoque point), then there are many results \cite{MK,GS,KKLL2017}. The
main reason is that we may give an criterion of the unique expansions. As
such we can discuss the set of points with unique expansions from the
symbolic point of view.

Representation of real numbers also arises naturally from Diophantine
approximation. Let $x$ be a real number. We say that $x$ is badly
approximable if there exists a positive integer $n$ such that for any
rational number $p/q$,
\begin{equation*}
|x-p/q|>(nq^{2})^{-1}.
\end{equation*}
Clearly, the set of badly approximable points is of Lebesgue measure zero.
However, Hall \cite{Hall} proved that every real number can be expressed as
the sum of two badly approximable numbers. For any $m\geq 2$, let $F(m)$
denote the set of numbers
\begin{equation*}
F(m)=\{[t,a_1,a_2,\cdots, ]:t\in \mathbb{Z}, 1\leq a_i\leq m \mbox{ for }
i\leq 1\}.
\end{equation*}
Hall \cite{Hall} proved that $F(4)+F(4)=\mathbb{R},$ where $A\pm B=\{x\pm
y:x\in A, y\in B\}$. Astels \cite{Astels} proved that
\begin{equation*}
F(2)+F(5)= \mathbb{R},F(2)-F(5)= \mathbb{R},F(3)-F(4)= \mathbb{R}.
\end{equation*}
There are many other related results for the arithmetic sum of $F(m)$, see
the references in \cite{Hall,Astels}.

The above two representations are prevalent in number theory. In this paper
we shall consider the so-called arithmetic representation of the real
numbers. First, we introduce some background of this representation. Given
two non-empty sets $A,B\subset \mathbb{R}$. Define $A\ast B=\{x\ast y:x\in
A,y\in B\}$, where $\ast $ is $+,-,\times $ or $\div $ (when $\ast =\div $, $%
y\neq 0$). We call $a=x\ast y,x\in A,y\in B$ an arithmetic representation in
terms of $A$ and $B$. Shortly, we may say $a=x\ast y$ is an arithmetic
representation if there is no fear of ambiguity. It is well-known that for
the middle-third Cantor set $C$,
\begin{equation*}
C-C=\{x-y:x,y\in C\}=[-1,1].
\end{equation*}%
Therefore, for any $a\in \lbrack -1,1]$, there are some $x,y\in C$ such that
$a=x-y$. The first proof of this result is due to Hugo Steinhaus \cite{HS}
in 1917. Recently, Athreya, Reznick, and Tyson \cite{Tyson} considered the
multiplication on the middle-third Cantor set, and proved that $17/21\leq \mathcal{L}(C\cdot C)\leq 8/9,$ where $\mathcal{L}$ denotes
the Lebesgue measure. Moreover, they also proved that the division on $C$,
denoted by $\dfrac{C}{C}$, is exactly the union of some closed sets. In \cite%
{XiKan1}, Tian et al. defined a class of overlapping self-similar sets as
follows: let $K$ be the attractor of the IFS
\begin{equation*}
\{f_{1}(x)=\lambda x,f_{2}(x)=\lambda x+c-\lambda ,f_{3}(x)=\lambda
x+1-\lambda \},
\end{equation*}%
where $f_{1}(I)\cap f_{2}(I)\neq \emptyset ,(f_{1}(I)\cup f_{2}(I))\cap
f_{3}(I)=\emptyset ,$ and $I=[0,1]$ is the convex hull of $K$. Then $K\cdot
K=[0,1]$ if and only if $(1-\lambda )^{2}\leq c$. Equivalently, they gave a
necessary and sufficient condition such that for any $x\in \lbrack 0,1]$
there exist some $y,z\in K$ such that $x=yz$.

Motivated by the multiple $%
\beta $-expansions, generally it is natural to analyze the set of points in $%
A\ast B=\{x\ast y:x\in A,y\in B\}$ such that these points have exactly $r$
different representations, i.e. we want to analyze the following set
\begin{equation}
U_{r}=\{x\in A\ast B:x\mbox{ has  exactly }r%
\mbox{  arithmetic
representations}\},  \label{U}
\end{equation}%
where $r\in \mathbb{N}^{+}$. In this paper, we assume the above algorithm $%
\ast $ is $+$ or $-$. For the classical middle-third Cantor set, take $%
1/3\in C-C=[-1,1]$, it is not difficult to prove that
$1/3$ has only three arithmetic representations in $C-C$, i.e.
\begin{equation*}
\dfrac{1}{3}=\dfrac{1}{3}-0=\dfrac{2}{3}-\dfrac{1}{3}=1-\dfrac{2}{3}.
\end{equation*}%
We may give a new explanation for Steinhaus' result from the projectional
perspective. Note that
$C-C$ is congruent to $\sqrt{2}\text{Proj}_{\theta }(C\times C)$,
where $\theta =\dfrac{3\pi }{4}$ and  Proj$_{\theta }$ denotes the orthogonal
projection onto $L_{\theta }$ which is  the line through
the origin in direction $\theta $. Since $C-C=[-1,1]$, it follows that
\begin{equation*}
\dim _{H}(\text{Proj}_{\theta }(C\times C))=\min \{\dim
_{H}(C)+\dim _{H}(C),1\}.
\end{equation*}%
In other words, the orthogonal projection of $C\times C$ to the line $y=-x$
does not drop the expected dimension. Indeed, similar result is still
correct for a general class of self-similar sets. Peres and Shmerkin \cite%
{PS}, Hochman and Shmerkin \cite{Hochman2012} proved the following elegant
result.

Let $K_1$ and $K_2$ be two self-similar sets with IFS's $\{f_i(x)=r_ix+a_i%
\}_{i=1}^{n}$ and $\{g_j(x)=r_j^{\prime}x+b_j\}_{j=1}^{m}$, respectively. If
 there are some  $r_i, r_j^{\prime}$ such that
\begin{equation*}
\dfrac{\log |r_i| }{\log |r_j^{\prime}|}\notin \mathbb{Q},
\end{equation*}
then
\begin{equation*}
\dim_{H}(K_1+K_2)=\min\{\dim_{H}(K_1)+\dim_{H}(K_2),1\},
\end{equation*}
and
$
\dim_{H}(K_1+K_2)=\dim_{P}(K_1+K_2)=\dim_{B}(K_1+K_2).
$
The condition in the above result is called the irrationality condition. The
result above indeed states that under the irrationality condition, the
Hausdorff dimension of the projection of two self-similar sets through the
angle $\pi/4$ does not decrease.

Now we go back to the middle-third Cantor set, and consider a slicing
problem, i.e. given $t\in \lbrack -1,1]$, then the set $U_{r}$ in (\ref{U}),
is
\begin{equation*}
U_{r}=\{t\in \lbrack -1,1]:\mathbf{Card}(\{y-x=t\}\cap (C\times C))=r\},
\end{equation*}%
where $r\in \mathbb{N}^{+}$. In other words, the multiple representational
problem is indeed a slicing problem in geometric measure theory. In this
paper, we shall consider the arithmetic addition or subtraction for more
than two Cantor sets. First, we give some basic definitions.

Suppose $n\geq 2$ and $\mathcal{A}_{i}\subset \{0,1,\cdots ,(n-1)\}$ for $%
i=1,\cdots ,l,$ let%
\begin{equation*}
K_{i}=\bigcup\nolimits_{a\in \mathcal{A}_{i}}\frac{K_{i}+a}{n}
\end{equation*}%
be self-similar sets contained in $[0,1].$ Fix $\mathbf{m}=(m_{1},\cdots
,m_{l})\in \mathbb{Z}^{l}$ with $\prod\nolimits_{i}m_{i}\neq 0$ denote
\begin{equation*}
S_{x}=\left\{ \mathbf{y}\in \prod\nolimits_{i=1}^{l}K_{i}:(\mathbf{m},%
\mathbf{y})=m_{1}y_{1}+\cdots +m_{l}y_{l}=x\right\} .
\end{equation*}

In this paper, we will focus on the fractal dimension of
\begin{equation*}
U_{r}=\{x:\text{\textbf{Card}}(S_{x})=r\}\text{ for }r<\infty ,
\end{equation*}%
It is worthwhile pointing out that if \textbf{Card}$(S_{x})=1$ then there is
a \emph{unique solution} for the equation
\begin{equation*}
x=m_{1}y_{1}+\cdots +m_{l}y_{l}\text{ with }y_{i}\in K_{i}\text{ }\forall i.
\end{equation*}

\medskip

Let $m_{\ast }=\sum\nolimits_{m_{i}<0}m_{i}$ and $m^{\ast
}=\sum\nolimits_{m_{i}>0}m_{i}$ where $\sum\nolimits_{b\in \emptyset }b=0.$
Given a subset $B$ of $\mathbb{R}^{l},$ we write $(\mathbf{m},B)=\cup _{b\in
B}(\mathbf{m},b).$ For $\mathbf{i}=(i_{1},\cdots ,i_{l})\in \prod_{i=1}^{l}%
\mathcal{A}_{i}$ and we write the small cube $\mathbf{c}_{\mathbf{i}}=\frac{%
(i_{1},\cdots ,i_{l})+[0,1]^{l}}{n}$. We say that the \textbf{covering
condition} holds for $\prod_{i=1}^{l}K_{i}$ with respect to $\mathbf{m}$, if
\begin{equation*}
\bigcup\nolimits_{\mathbf{i}\in \prod_{i=1}^{l}\mathcal{A}_{i}}(\mathbf{m},%
\mathbf{c}_{\mathbf{i}})=(\mathbf{m},[0,1]^{l}),
\end{equation*}%
where $(\mathbf{m},[0,1]^{l})=[m_{\ast },m^{\ast }]$ and $(\mathbf{m},%
\mathbf{c}_{\mathbf{i}})=n^{-1}((\mathbf{m},\mathbf{i})+[m_{\ast },m^{\ast
}]).$

We call $I=n^{-1}[u,u+1]\subset \lbrack m_{\ast },m^{\ast }]$ with $u\in
\mathbb{Z}$ an integer interval, and $\{J_{t}=[t,t+1]\}_{t\in \lbrack
m_{\ast },m^{\ast })\cap \mathbb{Z}}$ working intervals. We say $I$ is of
type $t$ (with respect to the small cube $c_{\mathbf{i}}$), if $u-(\mathbf{m}%
,\mathbf{i})=t\in \lbrack m_{\ast },m^{\ast }-1]\cap \mathbb{Z}$ for some $%
\mathbf{i}\in \prod_{i=1}^{l}\mathcal{A}_{i},$ i.e., $t$ is the relative
position of $I$ according to the projection interval $(\mathbf{m},\mathbf{c}%
_{\mathbf{i}})=n^{-1}((\mathbf{m},\mathbf{i})+[m_{\ast },m^{\ast }])$\ of
the small cube $\mathbf{c}_{\mathbf{i}}$. For $t\in \lbrack m_{\ast
},m^{\ast })\cap \mathbb{Z},$ the corresponding geometric type is $%
[0,1]^{l}\cap \{\mathbf{y}:(\mathbf{m},\mathbf{y})\in J_{t}\}$ or its
similar copy. Two integer intervals$\ I_{1}=n^{-1}[u_{1},u_{1}+1]$ and $%
I_{2}=n^{-1}[u_{2},u_{2}+1]$ are said to be congruent modulo $n,$ if $%
u_{1}\equiv u_{2}($mod $n),$ i.e., $I_{2}=I_{1}+k$ for some $k\in \mathbb{Z}%
. $

For a directed graph, we give a partial order on its strongly connected
components $\{H_{i}\}_{i}$, we denote $H\prec H^{\prime },$ if $H=H^{\prime
} $ or there is a directed path from one vertex of $H$ to another vertex of $%
H^{\prime }.$ Let $\rho (H_{i})$ be the spectral radius of the matrix with
respect to the subgraph restricted in $H_{i}.$ We say an infinite sequence $%
v_{i_{1}}v_{i_{2}}\cdots v_{i_{k}}v_{i_{k+1}}\cdots $\ of vertexes is
admissible if there is a directed edge from $v_{i_{k}}$ to $v_{i_{k+1}}$ for
all $k.$

In this paper, we need three directed graphs.

(1) The first graph has the vertex set of all integer intervals. For two
vertexes (or integer intervals)$\ I_{1}$ and $I_{2},$ there is a directed
edge from $I_{1}$ to $I_{2},$ denoted by $I_{1}\rightarrow I_{2},$ if and
only if there exists an integer $t\in \lbrack m_{\ast },m^{\ast }-1]$ such
that $I_{1}$ is of type $t$ and $I_{2}\subset J_{t}.$ We denote $%
I\rightarrowtail J_{t}$ if the integer interval $I$ is of type $t.$

(2) The second graph is a subgraph of the first. Let $\Xi $\ denote the
collection of integer intervals $I$ such that there is a unique $\mathbf{i}%
\in \prod_{i=1}^{l}\mathcal{A}_{i}$ satisfying $I\subset (\mathbf{m},\mathbf{%
c}_{\mathbf{i}}).$ Then we obtain a directed subgraph $G_{\Xi }$ of the
first one. From this directed graph, we obtain a $0$-$1$ transition matrix $%
M $ with its spectral radius $\rho (M)$. Denote by $\{\Xi _{i}\}_{i}$\ the
strongly connected components of $G_{\Xi }.$ In the graph $G_{\Xi },$ we say
that a strongly connected component $\Xi _{i}$ can reach $J_{t},$ if there
is a directed path in $G_{\Xi }$ from one vertex of $\Xi _{i}$ to another
vertex of type $t.$ We also say that $[m_{\ast },m^{\ast }]$ is dominated by
$\Xi $ for $d>0,$ if each $J_{t}$ can be reached by some $\Xi _{i}$ with$%
\frac{\log \rho (M_{i})}{\log n}\geq d.$

(3) The third graph $G^{\ast }$ contains $G_{\Xi }$. A subset $\omega $\ of $%
\Xi $ is said to be congruent, if any two of $\omega $ are congruent. For
congruent subset $\omega ,$ let
\begin{equation*}
\mathcal{D}(\omega )=\{t:\text{there is an element of }\omega \text{
contained in }J_{t}=[t,t+1]\}.
\end{equation*}%
Let the vertex set of $G^{\ast }$ be the collection of all congruent subsets
of $\Xi .$ Then there is a directed edge from $\omega $ to $\omega ^{\prime
},$ if and only if for any $I\in \omega ,$ there exists some $I^{\prime }\in
\omega ^{\prime }$ such that $I\rightarrow I^{\prime }$ (in the first
graph), and for any $I^{\prime }\in \omega ^{\prime },$ there exists some $%
I\in \omega $ such that $I\rightarrow I^{\prime }$ (in the first graph).
Denote by $\{\Omega _{j}\}_{j}$\ the strongly connected components of $%
G^{\ast }.$ Note that $G_{\Xi }$ is a subgraph of $G^{\ast }$, and any
strongly connected component of $G_{\Xi }$ is also a strongly connected
component of $G^{\ast },$ hence $\{\rho (\Xi _{i})\}_{i}\subset \{\rho
(\Omega _{j})\}_{j}.$

\begin{theorem}
\label{T:1}Let $s=\frac{\log \rho (M)}{\log n}.$ Then we have the following
results. \newline
(1) $\dim _{H}U_{1}\geq s$ and $\mathcal{H}^{s}(U_{1})>0$. \newline
(2) Suppose the covering condition holds for $\prod_{i=1}^{l}K_{i}$ with
respect to $\mathbf{m}$, then
\begin{equation*}
\dim _{H}U_{1}=s.
\end{equation*}%
(3) Suppose the covering condition holds for $\prod_{i=1}^{l}K_{i}$ with
respect to $\mathbf{m}$ and $K_{i}$ satisfies the strong separation
condition holds (i.e., $1\notin \mathcal{A}_{i}-\mathcal{A}_{i}$) for all $%
i. $ Suppose $s>0,$ then
\begin{equation*}
\mathcal{H}^{s}(U_{1})=\infty
\end{equation*}%
if and only if there are two different strongly connected components $\Xi
_{i}$ and $\Xi _{j}$ in $G_{\Xi }$ such that $\rho (\Xi _{i})=\rho (\Xi
_{j})=\rho (M)$ and $\Xi _{i}\prec \Xi _{j}.$ In particular if $M$ is
irreducible then
\begin{equation*}
0<\mathcal{H}^{s}(U_{1})<\infty .
\end{equation*}
\end{theorem}

\begin{figure}[tbph]
\centering\includegraphics[width=0.65\textwidth]{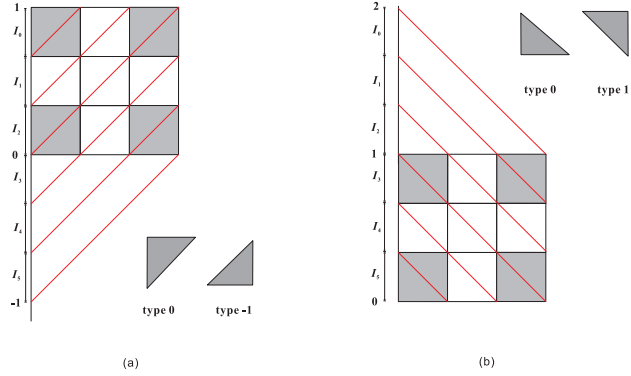} \vspace{-0.3cm}
\caption{Example 1}
\end{figure}

\begin{example}
As in part (a) of Fig. 1, let us consider the equation $x=-b_{1}+b_{2}$ with
$b_{i}\in C,$ where $C$ is the middle-third Cantor set and $(m_{1},m_{2})=(-1,1)$
with $m_{\ast }=-1$ and $m^{\ast }=1.$ Now, the covering condition and the
strong separation condition hold. Using the lines $\bigcup\nolimits_{d\in
3^{-1}\{-3,\cdots ,3\}}\{(x,y):y-x=d\},$ we obtain six integer intervals $%
I_{0},\cdots ,I_{5}$ and working intervals $[0,1]$ and $[-1,0].$ Four small
squares $\mathbf{c}_{(0,0)},\mathbf{c}_{(0,2)},\mathbf{c}_{(2,0)}$ and $%
\mathbf{c}_{(2,2)}$ are shown in part (a) of Fig. 1. We have $\Xi
=\{I_{0},I_{1},I_{4},I_{5}\}$ and the transition matrix
\begin{equation*}
M=\left(
\begin{array}{cccc}
1 & 1 & 0 & 0 \\
0 & 0 & 1 & 1 \\
1 & 1 & 0 & 0 \\
0 & 0 & 1 & 1%
\end{array}%
\right) \ \text{with }\rho (M)=2,
\end{equation*}%
then $\dim _{H}U_{1}=\frac{\log 2}{\log 3}=s$ and $0<\mathcal{H}%
^{s}(U_{1})<\infty $ since $M$ is irreducible. In this example, $\dim
_{H}U_{1}=\dim _{H}C=\frac{\log 2}{\log 3}.$

For part (b) of Fig. 1, let us consider the equation $x=b_{1}+b_{2}$ with $%
b_{i}\in C,$ where $C$ is the middle-third Cantor set and $(m_{1},m_{2})=(1,1)$
with $m_{\ast }=0$ and $m^{\ast }=2.$ We also have $\Xi
=\{I_{0},I_{1},I_{4},I_{5}\}$ and the same transition matrix $M$\ with $\rho
(M)=2.$ Hence in this case $\dim _{H}U_{1}=\dim _{H}C=\frac{\log 2}{\log 3}.$

In Examples 3 and 4 of Section 4, we have $\dim _{H}U_{1}<\min_{i}(\dim
_{H}K_{i})$ or $\dim _{H}U_{1}>\max_{i}(\dim _{H}K_{i}).$
\end{example}

Given $j\in \{0,\cdots ,(n-1)\},$ a $||\mathbf{m}||_{1}\times ||\mathbf{m}%
||_{1}$ matrix $T_{j}=(b_{uv})_{m_{\ast }\leq u,v\leq m^{\ast }-1}$ is
defined by
\begin{equation*}
b_{uv}=\text{\textbf{Card}}\{\mathbf{i}\in \prod\nolimits_{i=1}^{l}\mathcal{A%
}_{i}:(nu+j)-(\mathbf{m},\mathbf{i})=v\}.
\end{equation*}%
In fact, $b_{uv}>0$ if and only if $I_{(u,j)}\rightarrowtail J_{v},$ where $%
I_{(u,j)}=n^{-1}[nu+j,nu+j+1]$ and $J_{v}=[v,v+1].$ We also note that $%
b_{uv} $ is the number of small cubes $\mathbf{c}_{\mathbf{i}}$ such that $%
I_{(u,j)} $ has the relative position $v$ according to the projection
interval $(\mathbf{m},\mathbf{i})$\ of the small cube $\mathbf{c}_{\mathbf{i}%
}.$ Please see Examples 5 and 6 in Section 4 for this definition.

In fact, under the covering condition, by the method of \cite{LXZ}, we can
obtain that for Lebesgue almost all $x\in \lbrack m_{\ast },m^{\ast }],$
\begin{equation*}
\dim _{H}S_{x}=\dim _{B}S_{x}=\frac{\xi }{\log n},
\end{equation*}%
where $\xi $ is the Lyapunov exponent for the symmetric independent random
product of $T_{0},\cdots ,T_{n-1},$ i.e.,
\begin{equation*}
\xi =\lim_{k\rightarrow \infty }\frac{\log ||T_{x_{1}}T_{x_{2}}\cdots
T_{x_{k}}||}{\log k}
\end{equation*}%
with $x_{k}$ i.i.d. random variables assuming the values $\{0,1,\cdots
,(n-1)\}$ with equal probabilities. When the direction is fixed, some
general result on sections of self-similar sets can be found in \cite{WX}.

A mapping $\psi :\bigcup\nolimits_{I\in \Xi }$int($I)\rightarrow (m_{\ast
},m^{\ast })$ is defined by
\begin{equation*}
\psi (x)=n(x-a_{I})+t(I)\text{ for all }x\in \text{int(}%
I)=(a_{I},a_{I}+n^{-1}),
\end{equation*}%
where $t(I)$ is the type of $I.$ Then $\psi |_{\text{int}(I)}$ is a linear
surjection from $I$ to int($J_{t})=(t,t+1)$ with factor $n.$ Notice that if $%
x_{1},x_{2}\in \bigcup\nolimits_{I\in \Xi }$int($I)$ and $x_{1}=x_{2}+k$
with $k\in \mathbb{Z}$ and $x_{i}\in $int($I_{i}$) for $i=1,2,$ then
\begin{equation}
\psi (x_{1})-\psi (x_{2})=n(k-a_{I_{1}}+a_{I_{2}})+(t(I_{1})-t(I_{2}))\in
\mathbb{Z}\text{.}  \label{tex}
\end{equation}%
That means the integer intervals containing $\psi (x_{1})$ and $\psi (x_{2})$
are congruent (modulo $n)$ if $x_{1},x_{2}$ lie in the interiors of the
corresponding integer intervals respectively. In fact, the third graph
defined above is based on this observation. Given $z\in \lbrack
0,1)\backslash \{\frac{q_{1}}{n^{q_{2}}}\}_{q_{1},q_{2}\in \mathbb{Z}}$ and
a subset $\mathcal{D}$ of $\mathbb{Z}\cap \lbrack m_{\ast },m^{\ast }-1]$,
we say that the vector $(z+p)_{p\in \mathcal{D}}$ (with index lying in $%
\mathbb{Z}\cap \lbrack m_{\ast },m^{\ast }-1])$ has an infinite coding $%
\omega _{0}\omega _{1}\cdots \omega _{k}\cdots $ in $G^{\ast },$ if $\psi
^{i}(z+p)$ belongs to the interior of some integer interval of $\Xi $ for
each $i\geq 0$ and $p\in \mathcal{D}$, and $\omega _{i}$ is the smallest
congruent subset of $\Xi $ containing $\bigcup\nolimits_{p\in \mathcal{D}%
}\{\psi ^{i}(z+p)\}$ for each $i\geq 0$ such that $\omega _{0}\omega
_{1}\cdots \omega _{k}\cdots $ is admissible in $G^{\ast }$. Let $e_{i}$ be
the $i$-th one of the natural basis on $\mathbb{R}^{||\mathbf{m}||_{1}}$ for
$m_{\ast }\leq i\leq m^{\ast }-1.$

\begin{theorem}
\label{T:2}Suppose the covering condition holds and $K_{i}$ satisfies the
strong separation condition for each $i$. Then
\begin{equation*}
U_{r}\backslash \{\frac{q_{1}}{n^{q_{2}}}\}_{q_{1},q_{2}\in \mathbb{Z}}\neq
\emptyset ,
\end{equation*}%
if and only if there exist $i,j_{1},\cdots ,j_{k}$ such that
\begin{equation*}
r=||e_{i}T_{j_{1}}\cdots T_{j_{k}}||_{1}
\end{equation*}%
with $e_{i}T_{j_{1}}\cdots T_{j_{k}}=(\beta _{m_{\ast }},\cdots ,\beta
_{m^{\ast }-1})$ and $\mathcal{D}=\{p:\beta _{p}\neq 0\}$ satisfying
\begin{equation}
\{z\in \lbrack 0,1)\backslash \{\frac{q_{1}}{n^{q_{2}}}\}_{q_{1},q_{2}\in
\mathbb{Z}}:(z+p)_{p\in \mathcal{D}}\text{ has an infinite coding in }%
G^{\ast }\}\neq \emptyset .  \label{q}
\end{equation}%
The condition (\ref{q}) implies that the set $\{\bar{\omega}=\omega
_{i_{0}}\omega _{i_{1}}\cdots \omega _{i_{k}}\cdots :$ the infinite sequence
$\bar{\omega}\ $is admissible in $G^{\ast }$ with $\mathcal{D}(\omega
_{i_{0}})=\mathcal{D\}}$ is non-empty. If the latter set is uncountable,
then (\ref{q}) follows. In particular, if $\mathcal{D}(\omega _{i_{0}})=%
\mathcal{D}$ and $\omega _{i_{0}}\in \Omega _{i}$ such that $\rho (\Omega
_{j})>0$ with $\Omega _{i}\prec \Omega _{j}$ for some $\Omega _{j},$ then (%
\ref{q}) follows.
\end{theorem}

\begin{theorem}
\label{T:3}Suppose the covering condition holds$.$ Let $U_{r}\neq \emptyset $
with $r\geq 2$ and
\begin{equation*}
d_{r}=\dim _{H}U_{r}.
\end{equation*}%
(1) Then we have
\begin{equation*}
d_{r}\leq \frac{\log \rho (M)}{\log n}.
\end{equation*}%
(2) Moreover, we assume that $K_{i}$ satisfies the strong separation
condition for each $i$, then either $U_{r}$ is countable, or
\begin{equation*}
d_{r}\in \left( \bigcup\nolimits_{i}\max \left\{ \frac{\log \rho (\Omega
_{j})}{\log n}:\Omega _{i}\prec \Omega _{j}\right\} \right) \setminus \{0\}%
\text{ and }\mathcal{H}^{d_{r}}(U_{r})>0.
\end{equation*}%
(3) Suppose $K_{i}$ satisfies the strong separation condition for all $i$
and $[m_{\ast },m^{\ast }]$ is dominated by $\Xi $ for $d_{r}$, then $%
\mathcal{H}^{d_{r}}(U_{r})=\infty .$ In particular, if $M$ is irreducible
and the elements in $\Xi $ can reach every working interval, then
\begin{equation*}
\mathcal{H}^{d_{r}}(U_{r})=\infty .
\end{equation*}
\end{theorem}

\begin{example}
Let us consider $C-C$ again. We note that for $\omega =\{I_{0},I_{4}\}$ we
only find the edges from $\omega $ to $\{I_{0}\}$ and $\{I_{1}\},$ and the
edges from $\{I_{1},I_{5}\}$ to $\{I_{0}\}$ and $\{I_{1}\}.$ The effective
part of the matrix is the submatrix with respect to $G_{\Xi }.$ Then the
non-empty set $U_{r}$ is countable or has dimension $\log 2/\log 3.$ We also
check that $M$ in Example 1 is irreducible and the elements in $\Xi $ can
reach every working interval. By Theorem \ref{T:3}, $U_{r}$ is countable or $%
\dim _{H}U_{r}=\log 2/\log 3$ with infinite Hausdorff measure.

Now, we will find $r$ such that $U_{r}$ is non-empty. Ignoring a countable
set $\{\frac{q}{3^{p}}\}_{p,q\in \mathbb{Z}},$ we can take the infinite
coding with digit $\{0,1,2,3,4,5\}$\ to represent the number in $[-1,1]$
uniquely$,$ where digit $j$ represents $I_{j}.$ Now, the coding is not free,
we have the following rules: \newline
(1) If the current digit is $0,2$ or $4,$ then the next digit shall be taken
in $\{0,1,2\};$ \newline
(2) If the current digit is $1,3$ or $5,$ then the next digit shall be taken
in $\{3,4,5\}.$ \newline
Let
\begin{equation*}
\eta (i)=\left\{
\begin{array}{ll}
1 & \text{if }i=0,1,4\text{ or }5\text{,} \\
2 & \text{otherwise.}%
\end{array}%
\right.
\end{equation*}%
Then
\begin{equation*}
\mathbf{Card}\{(b_{1},b_{2})\in C\times
C:x=-b_{1}+b_{2}\}=\prod\nolimits_{k}\eta (i_{k}),
\end{equation*}%
if $x$ has coding $i_{0}i_{1}i_{2}\cdots i_{k}\cdots ,$ when we ignore a
countable set $\{\frac{q_{1}}{3^{q_{2}}}\}_{q_{1},q_{2}\in \mathbb{Z}}.$%
\begin{figure}[tbph]
\centering\includegraphics[width=0.7\textwidth]{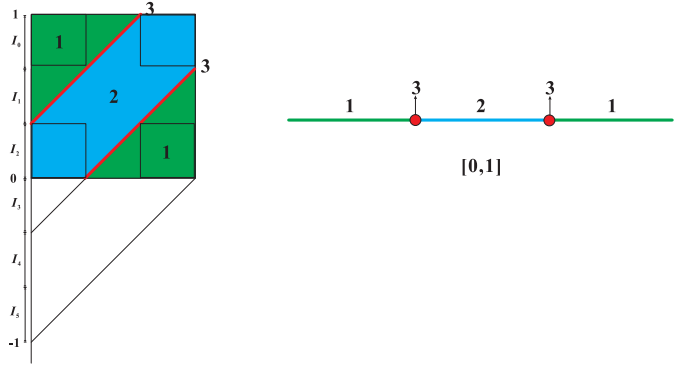} \vspace{-0.3cm}
\caption{Example 2}
\end{figure}

Let $P:\{0,1,2,3,4,5\}\rightarrow \{0,1,2\}$ be a projection defined by $%
P(2j)=P(2j+1)=j,$ we simplify this above model to the unit interval with $3$%
-adic expansion and take
\begin{equation*}
\bar{\eta}(j)=\left\{
\begin{array}{ll}
1 & \text{if }j=0,2\text{,} \\
2 & \text{if }j=1.%
\end{array}%
\right.
\end{equation*}%
Then the cardinality is $\prod\nolimits_{k}\bar{\eta}(j_{k})$ for the $3$%
-adic expansion $0.j_{0}j_{1}\cdots j_{k}\cdots .$ Using this projection, $%
U_{1}$ becomes the Cantor set, and $\{x:$\emph{\textbf{Card}}$(S_{x})=2^{u}\}
$ becomes the set
\begin{equation*}
\{0.j_{0}j_{1}\cdots j_{k}\cdots :\text{digit }1\text{ appears }u\text{
times in the expansion}\}.
\end{equation*}%
Let $s=\log 2/\log 3.$ Then one can check that%
\begin{equation}
\dim _{H}\{x:\mathbf{Card}(S_{x})=2^{u}\}=s,\text{ }\mathcal{H}^{s}\{x:%
\mathbf{Card}(S_{x})=2^{u}\}=\infty
\end{equation}%
for any $u\in \mathbb{N}^{+}$. Meanwhile, note red lines of $[0,1]^{2}$ and
red points of $[0,1]$ in Fig. 2, we have%
\begin{equation*}
U_{3\cdot 2^{u}}\subset \{\frac{q_{1}}{3^{q_{2}}}\}_{q_{1},q_{2}\in \mathbb{Z%
}}\text{ and \emph{\textbf{Card}}}(U_{3\cdot 2^{u}})=\aleph _{0},
\end{equation*}%
hence
\begin{equation}
\dim _{H}(U_{3\cdot 2^{u}})=0\text{ with }\mathcal{H}^{0}(U_{3\cdot
2^{u}})=\infty .
\end{equation}
\end{example}

\medskip

The paper is organized as follows. In Section 2, we provide the
preliminaries, including graph-directed construction \cite{Mauldin}, multi
dynamical systems for sections of self-similar sets \cite{Xi}\ and counting
formula of sections \cite{LXZ}. Section 3 is devoted to Theorem \ref{T:1} on
$U_{1}$. In Section 3, we will prove Theorems \ref{T:2} and \ref{T:3} on $%
U_{r}$ with $r\geq 2.$ We also give some examples in Section 4.

\bigskip

\section{Preliminaries}

Note that
\begin{equation}
\prod\nolimits_{i=1}^{l}K_{i}=\bigcup\nolimits_{\mathbf{i}\in \prod_{i=1}^{l}%
\mathcal{A}_{i}}\frac{\prod_{i=1}^{l}K_{i}+\mathbf{i}}{n}.  \label{aa}
\end{equation}%
For notational convenience, we write $S_{\mathbf{i}}(x)=\frac{x+\mathbf{i}}{n%
}$ and $S_{\mathbf{i}_{1}\cdots \mathbf{i}_{k}}=S_{\mathbf{i}_{1}}\circ
\cdots \circ S_{\mathbf{i}_{k}}.$ The cube $S_{\mathbf{i}_{1}\cdots \mathbf{i%
}_{k}}([0,1]^{l})$ with sidelength $n^{-k}\ $is said to be a basic cube of
rank $k.$ Denote the hyperplane
\begin{equation*}
H_{x}=\{\mathbf{y}\in \mathbb{R}^{l}:(\mathbf{m},\mathbf{y})=x\}.
\end{equation*}

\begin{lemma}
Suppose the covering condition holds for $\prod\nolimits_{i=1}^{l}K_{i}$
with respect to $\mathbf{m}.$ Then $H_{x}$ intersects $S_{\mathbf{i}%
_{1}\cdots \mathbf{i}_{k}}([0,1]^{l})$ if and only it intersects $S_{\mathbf{%
i}_{1}\cdots \mathbf{i}_{k}}(\prod_{i=1}^{l}K_{i}).$
\end{lemma}

\subsection{Graph-directed Construction}

$\ $

Recall the graph-directed construction introduced by Mauldin and Williams
\cite{Mauldin}.

Given a directed graph $G=(V,E)$, suppose $\{H_{i}\}_{i}$ are its strongly
connected components and $\mathcal{N}=(b_{v,v^{\prime }})_{v,v^{\prime }\in
V}$ is the matrix respect to the graph, i.e., $b_{v,v^{\prime }}$ is number
of edges from $v$ to $v^{\prime }.$ Given a directed edge $e$, we equip a
linear mapping $g_{e}:\mathbb{R}\rightarrow \mathbb{R}$ with contracting
ratio $n^{-1}$ where $n^{-1}$ is also the contracting ratio in (\ref{aa}).
By the classical result of \cite{Mauldin}, there is a unique family of
compact subsets $\{K_{v}\}_{v\in V}\subset \mathbb{R}$ such that
\begin{equation*}
K_{v}=\bigcup\nolimits_{v^{\prime }\in V}\bigcup\nolimits_{e\in \mathcal{E}%
(v,v^{\prime })}g_{e}(K_{v^{\prime }}),
\end{equation*}%
where $\mathcal{E}(v,v^{\prime })$ denotes the set of directed edges from $v$
to $v^{\prime }.$ We say that the open set condition holds, if there are
non-empty open sets $\{U_{v}\}_{v\in V}$ of $\mathbb{R}$ such that
\begin{equation*}
\bigcup\nolimits_{v^{\prime }\in V}\bigcup\nolimits_{e\in \mathcal{E}%
(v,v^{\prime })}g_{e}(U_{v^{\prime }})\subset U_{v}
\end{equation*}%
and the left hand of the above formula is a disjoint union for each $v\in V.$
By the results in \cite{Mauldin}, we have the following

\begin{lemma}
\label{L:s}Suppose the open set condition holds. Then \newline
(1) $\dim _{H}(\bigcup\nolimits_{v\in V}K_{v})=\frac{\log \rho (\mathcal{N})%
}{\log n}=s;$ \newline
(2) Suppose $s>0,$ then $\mathcal{H}^{s}(\bigcup\nolimits_{v\in V}K_{v})>0;$
\newline
(3) Suppose $s>0,$ then $\mathcal{H}^{s}(\bigcup\nolimits_{v\in
V}K_{v})=\infty $ if and only if there are two different strongly connected
components $H$ and $H^{\prime }$ such that $\rho (H)=\rho (H^{\prime })=\rho
(\mathcal{N})$ and $H\prec H^{\prime }.$\newline
(4) Suppose $v\in H_{i},$ then
\begin{equation*}
\dim _{H}K_{v}=\max \left\{ \frac{\log \rho (H_{j})}{\log n}:H_{i}\prec
H_{j}\right\} .
\end{equation*}%
Moreover, if $d=\dim _{H}K_{v}>0,$ then $\mathcal{H}^{d}(K_{v})>0.$
\end{lemma}

For example, let the symbolic system
\begin{equation*}
X=\{\bar{v}=v_{i_{0}}v_{i_{1}}\cdots v_{i_{k}}\cdots :\text{the infinite
sequence }\bar{v}\text{ is admissible}\}
\end{equation*}%
and $X_{v}$ the collection of infinite admissible sequences starting from $v$%
. There is a metric on $X$ defined by
\begin{equation*}
\text{d}(v_{i_{0}}\cdots v_{i_{k}}v_{i_{k+1}}\cdots ,v_{i_{0}}\cdots
v_{i_{k}}v_{i_{k+1}^{\prime }}\cdots )=n^{-k}\text{ if }v_{i_{k+1}}\neq
v_{i_{k+1}^{\prime }}.
\end{equation*}%
As in Lemma \ref{L:s}, by results of \cite{Mauldin}, we conclude that if $s=%
\frac{\log \rho (\mathcal{N})}{\log n}>0$ then $\dim _{H}(X)=s$ and $%
\mathcal{H}^{s}(X)>0$ and
\begin{equation*}
\dim _{H}X_{v}=\max \left\{ \frac{\log \rho (H_{j})}{\log n}:H_{i}\prec
H_{j}\right\} \text{ where }v\in H_{i}.
\end{equation*}%
We also have the following

\begin{claim}
If $\dim _{H}X_{v}=0,$ then $X_{v}$ is countable.
\end{claim}

\begin{proof}
It suffices to show for strongly connected component $H_{j}$ with $\rho
(H_{j})\leq 1$ the collection of admissible infinite sequences with letters
in $H_{j}$ is a finite set$.$ Suppose $\gamma $ is a\ Perron-Frobenius
eigenvector $\gamma $ of the matrix $\mathcal{N}_{j}$ w.r.t. $H_{j}$ such
that $||\gamma ||_{1}=1,$ and $\mathcal{D}_{k}$ denotes the set of
admissible sequences of length $k$ such that every letter lies in $H_{j}.$
Using the irreducibility, there exists a constant $\varsigma >0$ such that
for all $k,$ \textbf{Card}$(\mathcal{D}_{k})\leq ||\mathcal{N}%
_{j}^{k}(1,\cdots ,1)^{T}||_{1}\leq \varsigma ||\mathcal{N}_{j}^{k}\gamma
||_{1}=\varsigma ||\rho (H_{j})^{k}\gamma ||_{1}\leq \varsigma .$
\end{proof}

\subsection{Multi Dynamical System}

$\ $

We will construct a \textbf{multi dynamical system} including many expanding
maps from integer intervals to $[m_{\ast },m^{\ast }].$ In fact, for every
integer interval $I:n^{-1}[u,u+1]\subset \lbrack m_{\ast },m^{\ast }],$ we
equip the interval with several expanding maps with factor $n$ as follows.
Suppose $I$ is of type $t$ with $u-(\mathbf{m},\mathbf{i})=t,$ we let $f_{%
\mathbf{i}}:I\rightarrow J_{t}=[t,t+1]$ denote the corresponding linear
surjection in the form%
\begin{equation*}
f_{\mathbf{i}}(x)=nx-u+t:I\rightarrow J_{t}.
\end{equation*}%
Then the multi dynamical system consists of $\{f_{\mathbf{i}}\}_{\mathbf{i}%
\in \prod_{i=1}^{l}\mathcal{A}_{i}}.$

The next lemma reveals the connection between the multi dynamical system and
intersections of hyperplanes with basic cubes.

\begin{lemma}
If $x\in I=n^{-1}[u,u+1]$ ($u\in \mathbb{Z}$) and $\mathbf{i}\in
\prod_{i=1}^{l}\mathcal{A}_{i}$ such that $I\subset (\mathbf{m},\mathbf{c}_{%
\mathbf{i}})$, then
\begin{equation}
H_{x}\bigcap S_{\mathbf{i}}([0,1]^{l})=S_{\mathbf{i}}(H_{f_{\mathbf{i}%
}(x)}\cap \lbrack 0,1]^{l}).  \label{g}
\end{equation}
\end{lemma}

Using (\ref{g}) again and again, we obtain

\begin{lemma}
\label{L:why}If $x\in (\mathbf{m},\mathbf{c}_{\mathbf{i}_{1}})$ and $(f_{%
\mathbf{i}_{j}}\circ \cdots \circ f_{\mathbf{i}_{1}})(x)\in (\mathbf{m},%
\mathbf{c}_{\mathbf{i}_{j+1}})$ for all $1\leq j\leq k-1,$ then
\begin{equation*}
H_{x}\bigcap S_{\mathbf{i}_{1}\mathbf{i}_{2}\cdots \mathbf{i}%
_{k}}([0,1]^{l})=S_{\mathbf{i}_{1}\mathbf{i}_{2}\cdots \mathbf{i}%
_{k}}(H_{(f_{\mathbf{i}_{k}}\circ \cdots \circ f_{\mathbf{i}_{1}})(x)}\cap
\lbrack 0,1]^{l})
\end{equation*}
\end{lemma}

When $I\in \Xi ,$ since there is only one $\mathbf{i}\in \prod_{i=1}^{l}%
\mathcal{A}_{i}$ such that $I\subset (\mathbf{m},\mathbf{c}_{\mathbf{i}})$
and thus the type $t$ is uniquely determined by $I,$ we can write $f_{%
\mathbf{i}}$ to be $f_{I}$ for convenience$.$

If there is an admissible sequence $I_{i_{0}}I_{i_{1}}\cdots
I_{i_{k-1}}I_{i_{k}}\cdots $ of $G_{\Xi }\ $such that
\begin{equation}
x\in \text{int}(I_{i_{0}})\text{ and }g_{k}(x)=(f_{I_{i_{k-1}}}\circ \cdots
\circ f_{I_{i_{0}}})(x)\in \text{int}(I_{i_{k+1}})\text{ for all }k,
\label{tt}
\end{equation}%
then we say that $x\in \lbrack m_{\ast },m^{\ast }]$ has an \textbf{infinite
coding} $I_{i_{0}}I_{i_{1}}\cdots I_{i_{k-1}}I_{i_{k}}\cdots $ of $G_{\Xi }.$
By Lemma \ref{L:why} we have

\begin{lemma}
Let $g_{k}(x)=f_{I_{i_{k-1}}}\circ \cdots \circ f_{I_{i_{0}}}(x)$ be defined
as above. Then%
\begin{equation}
H_{x}\bigcap S_{\mathbf{i}_{1}\mathbf{i}_{2}\cdots \mathbf{i}%
_{k}}([0,1]^{l})=S_{\mathbf{i}_{1}\mathbf{i}_{2}\cdots \mathbf{i}%
_{k}}(H_{g_{k}(x)}\cap \lbrack 0,1]^{l}).  \label{g1}
\end{equation}
\end{lemma}

Let
\begin{equation*}
\Lambda =\{x\in \lbrack m_{\ast },m^{\ast }]:x\text{ has an infinite coding
in }\Xi \}.
\end{equation*}%
On the other hand, given an admissible sequence $I_{i_{0}}I_{i_{1}}\cdots
I_{i_{k-1}}I_{i_{k}}\cdots $ in $G_{\Xi },$ by the theorem of nested
interval, there exists a unique $x\in \lbrack m_{\ast },m^{\ast }]$ such that%
\begin{equation}
x\in I_{i_{0}}\text{ and }g_{k}(x)\in I_{i_{k+1}}\text{ for all }k.
\end{equation}%
That means
\begin{eqnarray*}
\bar{\Lambda}=\{x\in \lbrack m_{\ast },m^{\ast }]:
&&I_{i_{0}}I_{i_{1}}\cdots I_{i_{k-1}}I_{i_{k}}\cdots \text{ is admissible
in }G_{\Xi }\text{ such that} \\
&&x\in I_{i_{0}}\text{ and }g_{k}(x)\in I_{i_{k+1}}\text{ for all }k\}
\end{eqnarray*}%
and $\bar{\Lambda}\backslash \Lambda \subset \{\frac{q_{1}}{n^{q_{2}}}%
\}_{q_{1},q_{2}\in \mathbb{Z}}.$

We will consider a graph-directed construction induced by $\Xi .$ Suppose $%
I_{1},I_{2}\in \Xi $ with $I_{1}\rightarrow I_{2},$ the mapping
\begin{equation*}
g_{(I_{1},I_{2})}=(f_{I_{1}})^{-1}|_{I_{2}}:I_{2}\rightarrow I_{1}
\end{equation*}%
with contracting ratio $n^{-1}.$ Note that the open set condition holds for
this graph-directed construction. By the classical result of \cite{Mauldin},
there is a unique family of compact subsets $\{F_{I}\}_{I}$ such that $%
F_{I}\subset I$ for all $I\in \Xi $ and
\begin{equation*}
F_{I}=\bigcup\nolimits_{J\in \mathcal{E}(I)}g_{(I,J)}(F_{J}),
\end{equation*}%
where $\mathcal{E}(I)$ denotes the set of ending vertex of the edge starting
from $I.$ Now, we have%
\begin{equation}
\bar{\Lambda}=\bigcup\nolimits_{I\in \Xi }F_{I}\text{ with }\dim _{H}(\bar{%
\Lambda})=s.  \label{8}
\end{equation}%
Moreover, if $s>0$ then $\mathcal{H}^{s}\left( \bar{\Lambda}\right) >0$
under the open set condition$.$ Note that
\begin{equation*}
\bar{\Lambda}\backslash \Lambda \subset \{\frac{q_{1}}{n^{q_{2}}}%
\}_{q_{1},q_{2}\in \mathbb{Z}}
\end{equation*}%
and $\dim _{H}(\Lambda )\leq \dim _{H}(\bar{\Lambda}),$ by  Lemma \ref{L:s}, we have

\begin{lemma}
\label{L:1}Let $s=\frac{\log \rho (M)}{\log n}.$ We have \newline
(1) $\dim _{H}\Lambda =s;$ \newline
(2) If $s>0,$ then $\mathcal{H}^{s}(\Lambda )>0,$ moreover, $\mathcal{H}%
^{s}(\Lambda )=\infty $ if and only if there are two strongly connected
components $\Xi _{i}$ and $\Xi _{j}$ in $\Xi $ such that $\rho (M_{i})=\rho
(M_{j})=\rho (M)$ and $\Xi _{i}\prec \Xi _{j}.$ In particular, if $M$ is
irreducible, then $0<\mathcal{H}^{s}(\Lambda )<\infty .$
\end{lemma}

One can check the following lemma directly.

\begin{lemma}
\label{L:A}Suppose $\omega $ is a congruent subset of $\Xi $ and $\Omega
_{i} $ is the strongly connected component of $G^{\ast }$ containing $\omega
.$ Let
\begin{equation*}
A_{\omega }=\{z:\text{vector }(z+p)_{p\in \mathcal{D}(\omega )}\text{ has an
infinite coding in }G^{\ast }\text{ starting from }\omega \}.
\end{equation*}%
Then $d_{\omega }=\dim _{H}A_{\omega }=\max \left\{ \frac{\log \rho (\Omega
_{j})}{\log n}:\Omega _{i}\prec \Omega _{j}\right\} $. If $d_{\omega }>0,$
then%
\begin{equation*}
H^{d_{\omega }}(A_{\omega })>0.
\end{equation*}
\end{lemma}

\subsection{Counting Formula}

$\ $

For $z\in \lbrack 0,1),$ we let
\begin{equation*}
\alpha (z)=(\text{\textbf{Card}}(S_{z+m_{\ast }}),\text{\textbf{Card}}%
(S_{z+m_{\ast }+1}),\cdots ,\text{\textbf{Card}}(S_{z+m^{\ast }-1}))^{T},
\end{equation*}%
and%
\begin{equation*}
\sigma :[0,1)\rightarrow \lbrack 0,1)\text{\ such that\ }\sigma z=nz(\text{%
mod 1}).
\end{equation*}%
Let $\{x\}\in \lbrack 0,1)$ denote the decimal part of $x=[x]+\{x\},$ e.g. $%
\{-3.4\}=0.6$.

Using these notations and matrix $T_{j}$ in Section 1$\ $and results in \cite%
{LXZ}, we have

\begin{lemma}
\label{L:c}Suppose the covering condition holds. Let $x\in \lbrack m_{\ast
},m^{\ast }]\backslash \{\frac{q_{1}}{n^{q_{2}}}\}_{q_{1},q_{2}\in \mathbb{Z}%
}$. If%
\begin{equation*}
x=i+n^{-1}j_{1}+n^{-2}j_{2}+\cdots +n^{-k}j_{k}+\cdots ,
\end{equation*}%
i.e., $x-i$ has $n$-adic expansion $0.j_{1}j_{2}\cdots j_{k}\cdots ,$ then $%
||e_{i}T_{j_{1}}\cdots T_{j_{k}}||_{1}$ is the number of basic cubes of rank
$k$ intersecting the hyperplane $H_{x}.$ Moreover, if the strong separation
condition holds for each $K_{i}$, then
\begin{equation}
\text{\emph{\textbf{Card}}}(S_{x})=e_{i}T_{j_{1}}\cdots T_{j_{k}}\alpha
(\sigma ^{k}\{x\}).  \label{2222}
\end{equation}
\end{lemma}

\bigskip

\section{Dimension of Set with Unique Solution}

\begin{proof}[Proof of Theorem \protect\ref{T:1}]
$\ $

(1) We can show that $\dim _{H}U_{1}\geq s.$ By Lemma \ref{L:1}, we only
need to verify that%
\begin{equation}
\Lambda \backslash \{\frac{q_{1}}{n^{q_{2}}}\}_{q_{1},q_{2}\in \mathbb{Z}%
}\subset U_{1}.  \label{11}
\end{equation}%
In fact, suppose $x\in \Lambda \backslash \{\frac{q_{1}}{n^{q_{2}}}%
\}_{q_{1},q_{2}\in \mathbb{Z}}$ has an infinite coding $I_{i_{1}}\cdots
I_{i_{k}}\cdots $ in $G_{\Xi },$ then $g_{k}(x)\in $int($I_{i_{k+1}}$) for
all $k$ as in (\ref{tt})-(\ref{g1}). From the definition of coding and (\ref%
{g})-(\ref{g1}), we obtain a family of nested cubes $\{Q_{k}\}_{k}$\ such
that $Q_{k}$ is a basic cube of rank $k$ and $\left(
\prod_{i=1}^{l}K_{i}\right) \cap H_{x}\subset Q_{k}$ for all $k,$ which
implies that the intersection $\left( \prod_{i=1}^{l}K_{i}\right) \cap H_{x}$
is a singleton. Hence (\ref{11}) follows. It follows from (\ref{11}) and
Lemma \ref{L:1} that
\begin{equation*}
\mathcal{H}^{s}(U_{1})\geq \mathcal{H}^{s}(\Lambda )>0\text{ if }s>0.
\end{equation*}%
If $s=0,$ for getting $\mathcal{H}^{s}(U_{1})>0,$ we only need to show that $%
U_{1}$ is non-empty. In fact, we find \textbf{Card}$(S_{x})=1$ for $%
x=\sum\nolimits_{m_{i}<0}m_{i}(\min_{x_{i}\in
K_{i}}x_{i})+\sum_{m_{i}>0}m_{i}(\max_{x_{i}\in K_{i}}x_{i})$ which implies $%
U_{1}\neq \emptyset $.

(2) Suppose the covering condition holds. Assume that $x\notin \{\frac{q_{1}%
}{n^{q_{2}}}\}_{q_{1},q_{2}\in \mathbb{Z}}$ with \textbf{Card}$(S_{x})=1$
and $\mathbf{y}\in \prod_{i=1}^{l}K_{i}$ is the unique solution such that $(%
\mathbf{m},\mathbf{y})=x.$ Denote by $N_{k}$\ the number of basic cubes with
rank $k$ which contain $\mathbf{y}.$ By the covering condition and Lemma 1,
we obtain
\begin{equation*}
N_{k+1}\geq N_{k}\text{ for all }k.
\end{equation*}%
Note that $N_{k}\leq 2^{l}$ for all $k.$ Therefore, there is an integer $%
k_{0}$ such that $N_{k_{0}}=\max_{k}N_{k}.$ Fix a basic cube $S_{\mathbf{i}%
_{1}\cdots \mathbf{i}_{k_{0}}}([0,1]^{l})$ of rank $k_{0}$ containing $%
\mathbf{y}$, by Lemma \ref{L:why} we obtain that
\begin{equation*}
(f_{\mathbf{i}_{k_{0}}}\circ \cdots \circ f_{\mathbf{i}_{1}})(x)\text{ has
an infinite coding of }G_{\Xi },
\end{equation*}%
Hence
\begin{equation*}
U_{1}\subset \{x:(f_{\mathbf{i}_{k}}\circ \cdots \circ f_{\mathbf{i}%
_{1}})(x)\in \Lambda \text{ for some }f_{\mathbf{i}_{k}}\circ \cdots \circ
f_{\mathbf{i}_{1}}\},
\end{equation*}%
Since $\{f_{\mathbf{i}_{k}}\circ \cdots \circ f_{\mathbf{i}_{1}}\}_{\mathbf{i%
}_{1}\cdots \mathbf{i}_{k}}$ is a countable family, we obtain that
\begin{equation}
\dim _{H}U_{1}\leq \dim _{H}\Lambda .  \label{222}
\end{equation}%
It follows from (\ref{11})-(\ref{222}) and $\dim _{H}\Lambda =s$ (Lemma \ref%
{L:1}) that
\begin{equation*}
\dim _{H}U_{1}=s
\end{equation*}%
under the covering condition.

(3) Suppose the covering condition and the strong separation condition hold.
It follows from the strong separation condition that basic cubes of rank $k$
are pairwise disjoint. If $x\notin \{\frac{q_{1}}{n^{q_{2}}}%
\}_{q_{1},q_{2}\in \mathbb{Z}}$ and \textbf{Card}$(S_{x})=1$ with the unique
solution $\mathbf{y}$, then for each $k$ there is a unique basic cube $Q_{k}$
of rank $k$ containing $\mathbf{y}$ which means $x\in \Lambda $. Therefore%
\begin{equation}
\Delta (U_{1},\Lambda )\subset \{\frac{q_{1}}{n^{q_{2}}}\}_{q_{1},q_{2}\in
\mathbb{Z}},  \label{1111}
\end{equation}%
where $\Delta (A,B)=(A\backslash B)\cup (B\backslash A).$ The other part of
Theorem \ref{T:1}\ follows from Lemma \ref{L:1} and (\ref{1111}).
\end{proof}

\bigskip

\section{Number of Solutions}

\begin{proof}[Proof of Theorem \protect\ref{T:2}]
$\ $

Suppose
\begin{equation*}
x=i+n^{-1}j_{1}+n^{-2}j_{2}+\cdots +n^{-k}j_{k}+\cdots \in \lbrack m_{\ast
},m^{\ast }]\backslash \{\frac{q_{1}}{n^{q_{2}}}\}_{q_{1},q_{2}\in \mathbb{Z}%
}.
\end{equation*}%
Let \textbf{Card}$(S_{x})=r$ with solution set $\{\mathbf{y}_{1},\cdots ,%
\mathbf{y}_{r}\}.$ Denote by $N_{k}$\ the number of basic cubes of rank $k$
intersecting $\{\mathbf{y}_{1},\cdots ,\mathbf{y}_{r}\}.$ It follows from
the covering condition that%
\begin{equation*}
N_{k+1}\geq N_{k}\text{ for all }k.
\end{equation*}%
Note that $N_{k}\leq r$ for all $k$ due to the strong separation condition.
Then there exists an integer $k_{0}$ such that $N_{k_{0}}=r$ and thus
\begin{equation*}
\text{\textbf{Card}}(S_{x})=e_{i}T_{j_{1}}\cdots T_{j_{k_{0}}}\alpha (\sigma
^{k_{0}}\{x\})=||e_{i}T_{j_{1}}\cdots T_{j_{k_{0}}}||_{1}=r
\end{equation*}%
due to (\ref{2222}) in Lemma \ref{L:c}. Let $(\beta _{m_{\ast }},\cdots
,\beta _{m^{\ast }-1})=e_{i}T_{j_{1}}\cdots T_{j_{k_{0}}}$ and $\mathcal{D}%
=\{p:\beta _{p}\neq 0\}.$ Then $\sigma ^{k_{0}}\{x\}+p\in U_{1}$ for all $%
p\in \mathcal{D}$, that means $\sigma ^{k_{0}}\{x\}+p$ has an infinite
coding in $G_{\Xi }$ for all $p\in \mathcal{D}$. Hence
\begin{equation*}
\{z\in \lbrack 0,1)\backslash \{\frac{q_{1}}{n^{q_{2}}}\}_{q_{1},q_{2}\in
\mathbb{Z}}:(z+p)_{p\in \mathcal{D}}\text{ has an infinite coding in }%
G^{\ast }\}\neq \emptyset .
\end{equation*}

On the other hand, if $r=||e_{i}T_{j_{1}}\cdots T_{j_{k}}||_{1}$ with $%
e_{i}T_{j_{1}}\cdots T_{j_{k}}=(\beta _{m_{\ast }},\cdots ,\beta _{m^{\ast
}-1})$ and $\mathcal{D}=\{p:\beta _{p}\neq 0\}$ satisfying
\begin{equation*}
\{z\in \lbrack 0,1)\backslash \{\frac{q_{1}}{n^{q_{2}}}\}_{q_{1},q_{2}\in
\mathbb{Z}}:(z+p)_{p\in \mathcal{D}}\text{ has an infinite coding in }%
G^{\ast }\}\neq \emptyset .
\end{equation*}%
Then we can construct
\begin{equation*}
x=\left( i+\sum\nolimits_{k=1}^{k_{0}}n^{-k}j_{k}+n^{-k_{0}}z\right) \in
\lbrack 0,1)\backslash \{\frac{q_{1}}{n^{q_{2}}}\}_{q_{1},q_{2}\in \mathbb{Z}%
}
\end{equation*}%
and obtain that \textbf{Card}($S_{x})=r$ using (\ref{2222}) again.
\end{proof}

\medskip

\begin{proof}[Proof of Theorem \protect\ref{T:3}]
$\ $

(1) First we will show that $\dim _{H}U_{r}\leq \frac{\log \rho (M)}{\log n}%
. $

Let \textbf{Card}$(S_{x})=r$ with solution set $\{\mathbf{y}_{1},\cdots ,%
\mathbf{y}_{r}\}.$ Denote by $N_{k}$\ the number of basic cubes of rank $k$
intersecting $\{\mathbf{y}_{1},\cdots ,\mathbf{y}_{r}\}.$ It follows from
the covering condition that%
\begin{equation*}
N_{k+1}\geq N_{k}\text{ for all }k.
\end{equation*}%
Note that $N_{k}\leq 2^{l}r$ for all $k.$ Therefore, there exists an integer
$k_{0}$ such that $N_{k_{0}}=\max_{k}N_{k}.$ Fix a basic cube $S_{\mathbf{i}%
_{1}\cdots \mathbf{i}_{k_{0}}}([0,1]^{l})$ of rank $k_{0}$ intersecting $\{%
\mathbf{y}_{1},\cdots ,\mathbf{y}_{r}\}$, by Lemma \ref{L:why}, we obtain
that
\begin{equation*}
(f_{\mathbf{i}_{k_{0}}}\circ \cdots \circ f_{\mathbf{i}_{1}})(x)\text{ has
an infinite coding of }G_{\Xi }.
\end{equation*}%
Hence
\begin{equation*}
U_{r}\subset \{x:(f_{\mathbf{i}_{k}}\circ \cdots \circ f_{\mathbf{i}%
_{1}})(x)\in \Lambda \text{ for some }f_{\mathbf{i}_{k}}\circ \cdots \circ
f_{\mathbf{i}_{1}}\},
\end{equation*}%
Since $\{f_{\mathbf{i}_{k}}\circ \cdots \circ f_{\mathbf{i}_{1}}\}_{\mathbf{i%
}_{1}\cdots \mathbf{i}_{k}}$ is a countable family, we obtain
\begin{equation*}
\dim _{H}U_{r}\leq \dim _{H}\Lambda =\frac{\log \rho (M)}{\log n}.
\end{equation*}

\medskip

(2) We will obtain the dimension of $U_{r}.$ Suppose $U_{r}$ is uncountable.

Let $x=i+n^{-1}j_{1}+n^{-2}j_{2}+\cdots +n^{-k}j_{k}+\cdots \in \lbrack
m_{\ast },m^{\ast }]\backslash \{\frac{q_{1}}{n^{q_{2}}}\}_{q_{1},q_{2}\in
\mathbb{Z}}.$ Suppose $N_{k}$ is defined as above, by the above discussion,
there exists an integer $k_{0}$ such that $N_{k_{0}}=r$ and thus
\begin{equation}
\text{\textbf{Card}}(S_{x})=e_{i}T_{j_{1}}\cdots T_{j_{k_{0}}}\alpha (\sigma
^{k_{0}}\{x\})=||e_{i}T_{j_{1}}\cdots T_{j_{k_{0}}}||_{1}=r  \label{99}
\end{equation}%
due to (\ref{2222}) in Lemma \ref{L:c}. Let $(\beta _{m_{\ast }},\cdots
,\beta _{m^{\ast }-1})=e_{i}T_{j_{1}}\cdots T_{j_{k_{0}}}$ and $\mathcal{D}%
_{x}=\{p:\beta _{p}\neq 0\}.$ When $x$ is fixed, we use $\mathcal{D}$ to
replace $\mathcal{D}_{x}$ for notational convenience.

Then the above formula (\ref{99})\ implies that%
\begin{equation*}
\sigma ^{k_{0}}\{x\}+p\in U_{1}\text{ if }p\in \mathcal{D}\text{.}
\end{equation*}%
Let $\sigma ^{k_{0}}\{x\}\in n^{-1}[h_{x},h_{x}+1).$ Then $\sigma
^{k_{0}}\{x\}+p\in \Lambda $ has an infinite coding in $G_{\Xi }$ for any $%
p\in \mathcal{D}.$ Let
\begin{equation*}
\mathcal{F}_{\mathcal{D}}=\{\omega _{h}:\omega
_{h}=\{n^{-1}[np+h,np+h+1]\}_{p\in \mathcal{D}}\in \Xi \text{ with }h\in
\lbrack 0,(n-1)]\cap \mathbb{Z}\}.
\end{equation*}%
Suppose $\omega _{h}\in \mathcal{F}_{\mathcal{D}}\ $and $\Omega _{i(h,%
\mathcal{D})}$ is the strongly connected components containing $\omega _{h}.$
If $\sigma ^{k_{0}}\{y\}=z\in n^{-1}[h,h+1),$ using the formula in Lemma \ref%
{L:c} we have%
\begin{equation*}
\text{\textbf{Card}}(S_{y})=e_{i}T_{j_{1}}\cdots T_{j_{k_{0}}}\alpha
(z)=(\beta _{m_{\ast }},\cdots ,\beta _{m^{\ast }-1})\alpha (z)=r,
\end{equation*}%
i.e.,
\begin{equation*}
\{y=i+\sum\nolimits_{k=1}^{k_{0}}n^{-k}j_{k}+n^{-k_{0}}z:z\in \Lambda \cap
\lbrack 0,1)\text{ whenever }p\in \mathcal{D}\}\subset U_{r}.
\end{equation*}%
In fact, $\sigma ^{k_{0}}\{y\}+p\in \Lambda $ for all $p\in \mathcal{D}$
implies that $(z+p)_{_{p\in \mathcal{D}}}$ has an infinite coding in $%
G^{\ast }$ starting from $\omega _{h},$ i.e., $z\in A_{\omega _{h}},$ where $%
A_{\omega _{h}}$ is defined in Lemma \ref{L:A}$.$ Using Lemma \ref{L:A} we
have%
\begin{equation*}
\dim _{H}U_{r}\geq \dim _{H}A_{\omega _{h}}=\max \left\{ \frac{\log \rho
(\Omega _{j})}{\log n}:\Omega _{i(h,\mathcal{D})}\prec \Omega _{j}\right\}
\end{equation*}%
for any $\omega _{h}\in \mathcal{F}_{\mathcal{D}}=\mathcal{F}_{\mathcal{D}%
_{x}}.$ In particular, for $d^{\prime }=\max \left\{ \frac{\log \rho (\Omega
_{j})}{\log n}:\Omega _{i(h,\mathcal{D})}\prec \Omega _{j}\right\} ,$ using
Lemma \ref{L:A} again we have%
\begin{equation}
\mathcal{H}^{d^{\prime }}(U_{r})\geq (n^{-k_{0}})^{d^{\prime }}\mathcal{H}%
^{d^{\prime }}(A_{\omega _{h}})>0\text{ if }d^{\prime }>0.  \label{zzz}
\end{equation}

Let $\mathcal{G}=\{\mathcal{D}_{x}:$\textbf{Card}$(S_{x})=r\}.$ The above
discussion shows that
\begin{equation*}
\text{if }x\in U_{r}\backslash \{\frac{q_{1}}{n^{q_{2}}}\}_{q_{1},q_{2}\in
\mathbb{Z}},\text{ then }\sigma ^{k_{0}}\{x\}\in A_{\omega _{h_{x}}}\text{
for some }k_{0}.
\end{equation*}%
Hence
\begin{eqnarray*}
&&\dim _{H}(U_{r})=\dim _{H}\left( U_{r}\backslash \{\frac{q_{1}}{n^{q_{2}}}%
\}_{q_{1},q_{2}\in \mathbb{Z}}\right) \\
&\leq &\max_{\mathcal{D\in G}}\dim _{H}\left( \bigcup\nolimits_{\omega \in
\mathcal{F}_{\mathcal{D}}}A_{\omega }\right) \\
&\leq &\max_{\mathcal{D\in G}}\max_{\omega _{h}\in \mathcal{F}_{\mathcal{D}%
}}\max \left\{ \frac{\log \rho (\Phi _{j})}{\log n}:\Omega _{i(h,\mathcal{D}%
)}\prec \Omega _{j}\right\} .
\end{eqnarray*}%
i.e.,
\begin{equation}
d_{r}=\dim _{H}U_{r}=\max_{\mathcal{D\in G}}\max_{\omega _{h}\in \mathcal{F}%
_{\mathcal{D}}}\max \left\{ \frac{\log \rho (\Phi _{j})}{\log n}:\Omega
_{i(h,\mathcal{D})}\prec \Omega _{j}\right\} .  \label{a}
\end{equation}%
By (\ref{zzz}) and (\ref{a}), we obtain that
\begin{equation*}
\mathcal{H}^{d_{r}}(U_{r})>0\text{ if }d_{r}>0.
\end{equation*}

It suffices to show that $d_{r}>0.$ In fact, by the above discussion we have
\begin{equation}
\mathcal{U}_{r}\subset \bigcup\limits_{\mathcal{D\in G}}\bigcup\limits_{%
\omega _{h}\in \mathcal{F}_{\mathcal{D}}}\bigcup\limits_{(i,j_{1},\cdots
j_{k_{0}})}\{y=i+\sum\nolimits_{k=1}^{k_{0}}n^{-k}j_{k}+n^{-k_{0}}z:z\in
A_{\omega _{h}}\},  \label{abb}
\end{equation}%
where $\mathcal{U}_{r}=U_{r}\backslash \{\frac{q_{1}}{n^{q_{2}}}%
\}_{q_{1},q_{2}\in \mathbb{Z}}.$ Note that
\begin{equation*}
d_{r}=\max_{\mathcal{D\in G}}\max_{\omega _{h}\in \mathcal{F}_{\mathcal{D}%
}}\dim _{H}(A_{\omega _{h}}).
\end{equation*}%
If $d_{r}=0,$ then $\dim _{H}(A_{\omega _{h}})=0$ for all $\omega _{h}$
which implies that $A_{\omega _{h}}$ is countable for all $\omega _{h}$ due
to Claim 1. Using (\ref{abb}) we obtain that $U_{r}$ is countable, this is a
contradiction.

\medskip

(3) Now we will obtain the infinity of Hausdorff measure $\mathcal{H}%
^{d_{r}}(U_{r})$ when
\begin{equation*}
\mathcal{U}_{r}=U_{r}\backslash \{\frac{q_{1}}{n^{q_{2}}}\}_{q_{1},q_{2}\in
\mathbb{Z}}\neq \emptyset \text{ with }r\geq 2
\end{equation*}%
and $[m_{\ast },m^{\ast }]$ is dominated by $\Xi $ for $d_{r}.$ Take $x\in
\mathcal{U}_{r}$ and $\omega _{h}\in \mathcal{F}_{\mathcal{D}_{x}}$ such that%
\begin{equation*}
\max \left\{ \frac{\log \rho (\Omega _{j})}{\log n}:\Omega _{i(h,\mathcal{D}%
)}\prec \Omega _{j}\right\} =d_{r},
\end{equation*}%
where $\Omega _{i(h,\mathcal{D})}$ is the strongly connected component
containing $\omega _{h}.$ It follows from Lemma \ref{L:A} that%
\begin{equation*}
\mathcal{H}^{d_{r}}(A_{\omega _{h}})>0.
\end{equation*}%
Since $[m_{\ast },m^{\ast }]$ is dominated by $\Xi $ for $d_{r},$ we note
that for $x=i+\sum\nolimits_{k=1}^{k_{0}}n^{-k}j_{k}+\cdots \in J_{i}$ as
above$,$ we take a strongly connected component $\Xi _{j}$\ of $G_{\Xi }$
with $\frac{\log \rho (\Xi _{j})}{\log n}\geq d_{r}$ such that there is a
path $\mathbf{P}^{\prime }$ in $G_{\Xi }$ from some $I\in \Xi _{j}$ to an
integer interval $I^{\prime }\in \Xi $ of type $i.$ Let
\begin{equation*}
\mathcal{C}_{k}=\{\mathbf{P}:\text{ }\mathbf{P}\text{ is a path in }\Xi _{j}%
\text{ ending at }I\text{ and }|\mathbf{P}|=k\},
\end{equation*}%
where $|\mathbf{P}|$ denotes the length of the path. Then we obtain a family
$\{\mathcal{B}_{k}\}_{k}$ of pairwise disjoint subsets of $U_{r}$ as follows
\begin{equation*}
\mathcal{B}_{k}=\left\{ \text{ }\mathbf{PP}^{\prime
}(i+\sum\nolimits_{k=1}^{k_{0}}n^{-k}j_{k}+n^{-k_{0}}z):\mathbf{P}\in
\mathcal{C}_{k}\text{ and }z\in A_{\omega _{h}}\right\} \text{ for }%
k=1,2,\cdots ,
\end{equation*}%
where the coding $\mathbf{PP}^{\prime }x_{1}$ denotes a real number which is
turned\ to be $x_{1}\in \lbrack t,t+1]$ through a series of expanding maps
with respect to the path $\mathbf{PP}^{\prime }$ in $\Xi $ and the expanding
map according to $I^{\prime }\rightarrowtail J_{i}.$ Hence
\begin{equation}
\mathcal{H}^{d_{r}}(\mathcal{B}_{k})\geq \text{\textbf{Card}}(\mathcal{C}%
_{k})(n^{-(k+k_{0})d_{r}}\mathcal{H}^{d_{r}}(A_{\omega _{h}})),  \label{gggg}
\end{equation}%
and there is a constant $c>0$ such that%
\begin{equation}
\text{\textbf{Card}}(\mathcal{C}_{k})\geq c\rho (\Xi _{j})^{k}\geq
cn^{d_{r}k}.  \label{22}
\end{equation}%
Here we can check \textbf{Card}$(\mathcal{C}_{k})\geq c\rho (\Xi _{j})^{k}$
in (\ref{22}), because by the irreducibility we can take a\ Perron-Frobenius
eigenvector $\gamma $ of the matrix $M_{j}$ w.r.t. $\Xi _{j}$ such that $%
||\gamma ||_{1}=1,$ and thus
\begin{equation*}
\text{\textbf{Card}}(\mathcal{C}_{k})\geq c||M_{j}^{k}(1,\cdots
,1)^{T}||_{1}\geq c||M_{j}^{k}\gamma ||_{1}=c||\rho (\Xi _{j})^{k}\gamma
||_{1}\geq c\rho (\Xi _{j})^{k}.
\end{equation*}%
It follows from (\ref{gggg}) and (\ref{22}) that there is a constant $%
c^{\prime }>0$ such that%
\begin{equation*}
\mathcal{H}^{d_{r}}(\mathcal{B}_{k})\geq c^{\prime }\text{ for all }k.
\end{equation*}%
Hence
\begin{equation*}
\mathcal{H}^{d_{r}}(U_{r})\geq \mathcal{H}^{d_{r}}\left( \bigcup_{k}\mathcal{%
B}_{k}\right) \geq \sum\nolimits_{k}\mathcal{H}^{d_{r}}\left( \mathcal{B}%
_{k}\right) =+\infty .
\end{equation*}
\end{proof}

\bigskip

\section{Examples}

In this section, we give Examples 3 and 4 to illustrate Theorem \ref{T:1}
and Examples 5 and 6 to interpret the definition of the matrices $%
\{T_{j}\}_{j=0}^{n-1}.$

\begin{figure}[tbph]
\centering\includegraphics[width=0.7\textwidth]{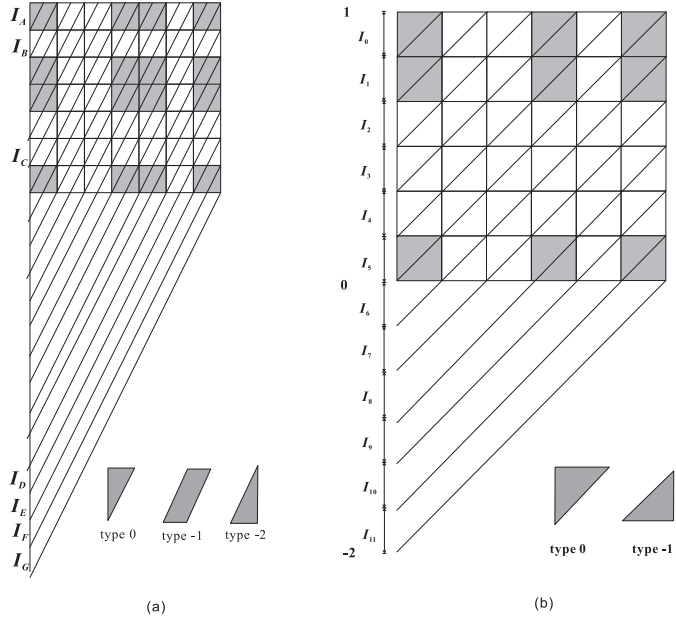} \vspace{-0.3cm}
\caption{Examples 3 and 4}
\end{figure}

\begin{example}
As in part (a) of Fig. 3, let us consider the equation $x=-2b_{1}+b_{2}$
with $b_{i}\in K=\cup _{a\in \{0,3,4,6\}}\frac{K+a}{7},$ where $%
(m_{1},m_{2})=(-2,1)$ with $m_{\ast }=-2$ and $m^{\ast }=1.$ The covering
condition holds. We have $\Xi =\{I_{A},I_{B},I_{C},I_{D},I_{E},I_{F},I_{G}\}$
and the transition matrix
\begin{equation*}
M=\left(
\begin{array}{ccccccc}
1 & 1 & 1 & 0 & 0 & 0 & 0 \\
0 & 0 & 0 & 0 & 0 & 0 & 0 \\
0 & 0 & 0 & 1 & 1 & 1 & 1 \\
0 & 0 & 0 & 1 & 1 & 1 & 1 \\
1 & 1 & 1 & 0 & 0 & 0 & 0 \\
0 & 0 & 0 & 0 & 0 & 0 & 0 \\
0 & 0 & 0 & 1 & 1 & 1 & 1%
\end{array}%
\right) \ \text{with }\rho (M)=\frac{\sqrt{5}+3}{2}.
\end{equation*}%
Then $\dim _{H}U_{1}=\frac{\log \frac{\sqrt{5}+3}{2}}{\log 7}=s$. In this
example, the spectral radius is not an integer and $\dim _{H}U_{1}<\dim
_{H}K=\frac{\log 4}{\log 7}.$
\end{example}

\begin{example}
As in part (b) of Fig. 3, let us consider the equation $x=-b_{1}+b_{2}$ with
$b_{1}\in K_{1}=\cup _{a\in \{0,3,5\}}\frac{K_{1}+a}{6}$ and $b_{2}\in
K_{2}=\cup _{a\in \{0,4,5\}}\frac{K_{2}+a}{6},$ where $(m_{1},m_{2})=(-1,1)$
with $m_{\ast }=-1$ and $m^{\ast }=1.$ The covering condition and the strong
separation condition hold. We have $\Xi
=\{I_{0},I_{2},I_{3},I_{7},I_{8},I_{9},I_{10},I_{11}\}$ and the transition
matrix
\begin{equation*}
M=\left(
\begin{array}{cccccccc}
1 & 1 & 1 & 0 & 0 & 0 & 0 & 0 \\
0 & 0 & 0 & 1 & 1 & 1 & 1 & 1 \\
1 & 1 & 1 & 0 & 0 & 0 & 0 & 0 \\
0 & 0 & 0 & 1 & 1 & 1 & 1 & 1 \\
1 & 1 & 1 & 0 & 0 & 0 & 0 & 0 \\
0 & 0 & 0 & 1 & 1 & 1 & 1 & 1 \\
1 & 1 & 1 & 0 & 0 & 0 & 0 & 0 \\
0 & 0 & 0 & 1 & 1 & 1 & 1 & 1%
\end{array}%
\right) \ \text{with }\rho (M)=4.
\end{equation*}%
Then $\dim _{H}U_{1}=\frac{\log 4}{\log 6}=s$ and $0<\mathcal{H}%
^{s}(U_{1})<\infty $ since $M$ is irreducible. Notice that in this example $%
\dim _{H}U_{1}>\max \{\dim _{H}K_{1},\dim _{H}K_{2}\}=\frac{\log 3}{\log 6}.$
\end{example}

\begin{figure}[tbph]
\centering\includegraphics[width=0.75\textwidth]{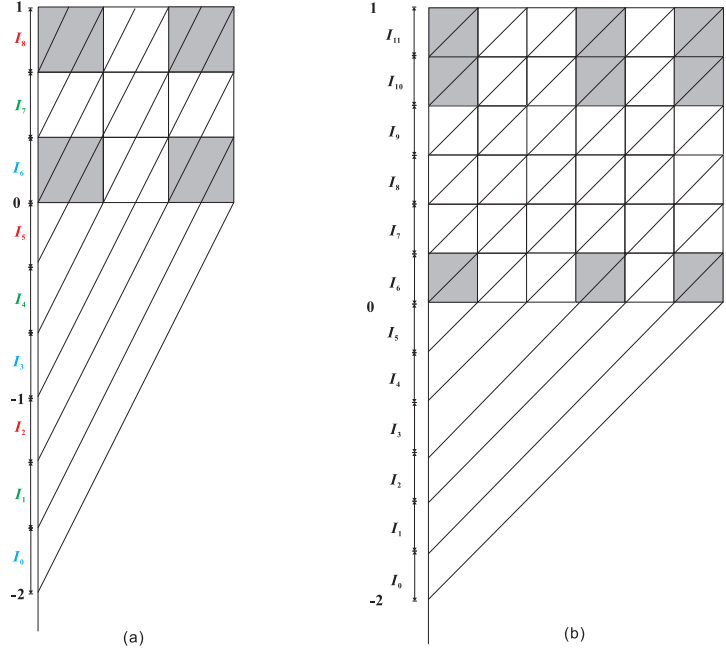} \vspace{-0.3cm}
\caption{Examples 5 and 6}
\end{figure}

\begin{example}
Consider $-2C+C$ as in part (a) of Fig. 4. We have
\begin{equation*}
T_{0}=%
\begin{array}{cc}
&
\begin{array}{ccc}
\!J_{-2} & \!\!J_{-1} & \!\!\!J_{0}%
\end{array}
\\
\begin{array}{c}
I_{0} \\
I_{3} \\
I_{6}%
\end{array}
& \left(
\begin{array}{ccc}
1 & \text{ }0 & \text{ }0 \\
0 & \text{ }1 & \text{ }0 \\
1 & \text{ }0 & \text{ }1%
\end{array}%
\right)%
\end{array}%
,T_{1}=%
\begin{array}{cc}
&
\begin{array}{ccc}
\!J_{-2} & \!\!J_{-1} & \!\!\!J_{0}%
\end{array}
\\
\begin{array}{c}
I_{1} \\
I_{4} \\
I_{7}%
\end{array}
& \left(
\begin{array}{ccc}
0 & \text{ }1 & \text{ }0 \\
1 & \text{ }0 & \text{ }1 \\
0 & \text{ }1 & \text{ }0%
\end{array}%
\right)%
\end{array}%
\end{equation*}%
and $T_{2}=%
\begin{array}{cc}
&
\begin{array}{ccc}
\!J_{-2} & \!\!J_{-1} & \!\!\!J_{0}%
\end{array}
\\
\begin{array}{c}
I_{2} \\
I_{5} \\
I_{8}%
\end{array}
& \left(
\begin{array}{ccc}
1 & \text{ }0 & \text{ }1 \\
0 & \text{ }1 & \text{ }0 \\
0 & \text{ }0 & \text{ }1%
\end{array}%
\right)%
\end{array}%
.$
\end{example}

\begin{example}
Consider $-K_{1}+K_{2}$ for $K_{1}=\cup _{a\in \{0,3,5\}}\frac{K_{1}+a}{6}$
and $K_{2}=\cup _{a\in \{0,4,5\}}\frac{K_{2}+a}{6}$\ as in part (b) of Fig.
4. We have
\begin{eqnarray*}
T_{0} &=&%
\begin{array}{cc}
&
\begin{array}{cc}
\!\!J_{-1} & \!\!\!J_{0}%
\end{array}
\\
\begin{array}{c}
I_{0} \\
I_{6}%
\end{array}
& \left(
\begin{array}{cc}
1 & \text{ }0 \\
1 & \text{ }2%
\end{array}%
\right)%
\end{array}%
,\text{ }T_{1}=%
\begin{array}{cc}
&
\begin{array}{cc}
\!\!J_{-1} & \!\!\!J_{0}%
\end{array}
\\
\begin{array}{c}
I_{1} \\
I_{7}%
\end{array}
& \left(
\begin{array}{cc}
0 & \text{ }1 \\
1 & \text{ }1%
\end{array}%
\right)%
\end{array}%
, \\
T_{2} &=&%
\begin{array}{cc}
&
\begin{array}{cc}
\!\!J_{-1} & \!\!\!J_{0}%
\end{array}
\\
\begin{array}{c}
I_{2} \\
I_{8}%
\end{array}
& \left(
\begin{array}{cc}
1 & \text{ }0 \\
0 & \text{ }1%
\end{array}%
\right)%
\end{array}%
,\text{ }T_{3}=%
\begin{array}{cc}
&
\begin{array}{cc}
\!\!J_{-1} & \!\!\!J_{0}%
\end{array}
\\
\begin{array}{c}
I_{3} \\
I_{9}%
\end{array}
& \left(
\begin{array}{cc}
0 & \text{ }1 \\
1 & \text{ }0%
\end{array}%
\right)%
\end{array}%
, \\
T_{4} &=&%
\begin{array}{cc}
&
\begin{array}{cc}
\!\!J_{-1} & \!\!\!J_{0}%
\end{array}
\\
\begin{array}{c}
I_{4} \\
I_{10}%
\end{array}
& \left(
\begin{array}{cc}
1 & \text{ }0 \\
1 & \text{ }1%
\end{array}%
\right)%
\end{array}%
,\text{ }T_{5}=%
\begin{array}{cc}
&
\begin{array}{cc}
\!\!J_{-1} & \!\!\!J_{0}%
\end{array}
\\
\begin{array}{c}
I_{5} \\
I_{11}%
\end{array}
& \left(
\begin{array}{cc}
2 & \text{ }1 \\
0 & \text{ }1%
\end{array}%
\right)%
\end{array}%
.
\end{eqnarray*}
\end{example}

\bigskip

\section{Final remarks}

The covering condition is essential in this paper. Without this condition,
it is much more difficult to analyze the set $U_{r}$. We shall discuss this
case in another paper. Our main ideas may be implemented in the setting of
some overlapping self-similar sets. Nevertheless, the discussion is more
complicated. We consider only the addition or subtraction on the Cantor
sets. It is natural to consider similar problems for the multiplication or
division on self-similar sets, for instance, the set of points with unique
representations. Finally, our main results in this paper can be restated
from the slicing point of view.

\end{document}